\documentclass[dvipdfm,12pt]{amsart}
\usepackage{epstopdf, amsxtra, pifont}
\usepackage[dvips]{graphicx,color}
\newtheorem{state}{Proposition}
\newtheorem{example}{Example}

\newtheorem{corollary}{Corollary}

\newtheorem{lemma}{Lemma}
\newtheorem{definition}{Definition}
\newcommand{\la}{\boldsymbol{\lambda}}
\newcommand{\ka}{\boldsymbol{\kappa}}
\newcommand{\bea}{\begin{eqnarray}}
\newcommand{\eea}{\end{eqnarray}}
\newcommand{\be}{\begin{eqnarray*}}
\newcommand{\ee}{\end{eqnarray*}}
\definecolor{gold}{rgb}{0.75,0.55,0}
\definecolor{GreenYellow}{named}{GreenYellow}
\definecolor{Yellow}{named}{Yellow}
\definecolor{Goldenrod}{named}{Goldenrod}
\definecolor{Dandelion}{named}{Dandelion}
\definecolor{Apricot}{named}{Apricot}
\definecolor{Peach}{named}{Peach}
\definecolor{Melon}{named}{Melon}
\definecolor{YellowOrange}{named}{YellowOrange}
\definecolor{Orange}{named}{Orange}
\definecolor{BurntOrange}{named}{BurntOrange}
\definecolor{Bittersweet}{named}{Bittersweet}
\definecolor{RedOrange}{named}{RedOrange}
\definecolor{Mahogany}{named}{Mahogany}
\definecolor{Maroon}{named}{Maroon}
\definecolor{BrickRed}{named}{BrickRed}
\definecolor{Red}{named}{Red}
\definecolor{OrangeRed}{named}{OrangeRed}
\definecolor{RubinRed}{named}{RubineRed}
\definecolor{WildStrawberry}{named}{WildStrawberry}
\definecolor{Salmon}{named}{Salmon}
\definecolor{CarnationPink}{named}{CarnationPink}
\definecolor{Magenta}{named}{Magenta}
\definecolor{VioletRed}{named}{VioletRed}
\definecolor{Plum}{named}{Plum}
\definecolor{YellowGreen}{named}{YellowGreen}
\definecolor{RawSienna}{named}{RawSienna}
\definecolor{Sepia}{named}{Sepia}
\definecolor{Brown}{named}{Brown}
\definecolor{Tan}{named}{Tan}
\definecolor{Black}{named}{Black}
\definecolor{White}{named}{White}
\begin{document}
\title[Noncommutative Monomials]{Noncommutative Analogs of Monomial Symmetric Functions, Cauchy Identity and Hall Scalar Product.}
\maketitle
\begin{center}
\author{Lenny Tevlin} \\
\address{Physics Department \\
Yeshiva University \\
500 West 185th Street, \\
New York, N.Y. 10033, USA \\
}
\email{tevlin@yu.edu}
\end{center}

\begin{abstract}
This paper will introduce noncommutative analogs of monomial symmetric functions and
fundamental noncommutative symmetric functions. The expansion of
ribbon Schur functions in both of these basis is nonnegative. With these functions at hand, one can derive
a noncommutative Cauchy identity as well as study a noncommutative scalar product implied by Cauchy identity. This scalar product
seems to the noncommutative analog of Hall scalar product in the commutative theory.
\end{abstract}
\tableofcontents

\section{Introduction.}
\label{sec-intro}

The algebra of noncommutative symmetric functions \textbf{\textsf{NSym}} were introduced in \cite{GL} as a Hopf-dual to that of quasi-symmetric functions
\textbf{\textsf{QSym}}. From this point of view it is entirely natural that generalizations of
some of the bases of classical commutative symmetric functions \cite{M} would exist on the quasi-symmetric side, while others on the
 \textbf{\textsf{NSym}}
side. In particular, monomial symmetric functions $ m_{\la} $ were generalized in \cite{G}  to the quasi-symmetric monomials $ M_I $.
 Noncommutative complete symmetric functions $ S^I $, their dual, were introduced in \cite{GL}.
 Likewise, fundamental quasi-symmetric functions $ L_I $ form another basis of \textbf{\textsf{QSym}}
while ribbon Schur functions $ R^I $ form a dual basis in \textbf{\textsf{NSym}}\footnote{A word about notation convention for different symmetric
functions. I will use the standard notations of \cite{M} for classical symmetric functions, that of \cite{GL} for the rest, with two exceptions.
I will denote fundamental quasi-symmetric functions  $L_I$ following \cite{S} and use $R^I$ with a superscript for ribbon Schur functions. For all newly
introduced objects I will extend the above notation by using the same letter with super- (sub-) script for noncommutative (quasi-symmetric) functions.}.

The purpose of this paper is to show that nevertheless natural analogs of both monomial and fundamental bases exist on \textbf{\textsf{NSym}} side along
with quasi-symmetric analogs of complete and ribbon Schur bases.

The table below displays how functions to be introduced in this paper (marked in bold) fit in the known world.
\begin{center}
\begin{tabular}{lcrr}
& \textbf{\textsf{Sym}} & \textbf{\textsf{QSym}} & \textbf{\textsf{NSym}} \\
 monomial &  $ m_{\la} $ &  $ M_I $  & $ \mathbf{M^I} $  \\
& & & \\
 power sums & $ p_{\la} $  & $\mathbf{\Psi_I} $ &  $ \Psi^I $ \\
& & & \\
elementary &   $ e_{\la} $  & $\mathbf{\Lambda_I} $ &  $ \Lambda^I $  \\
 & & & \\
complete  & $ h_{\la} $  & $ \mathbf{S_I} $ & $ S^I $ \\
& & & \\
Schur & $ s_{\la} $  & $ \mathbf{R_I} $ &  $ R^I $ \\
& & & \\
fundamental & & $L_I $  &  $ \mathbf{L^I} $ \\
\end{tabular}
\end{center}

\section{Notations.}
\label{sec-notations}
\subsection{Compositions and Partitions.}
\label{sec-compart}
Let   $ I = (i_1, \ldots, i_n)$ be a composition, i.e. an ordered set of positive integers $ (i_1, \ldots, i_n) $, called
parts of the composition $ I $.
The sum of all components of the composition, its weight, is denoted by $|I|$
and the number of parts in the composition -- by $ \ell(I) $.

For a composition $ I $ one defines a \textbf{reverse} composition $ \bar{I} = (i_n, \ldots, i_1)$.

Two types of multiplication  were defined in \cite{GL}.  For two compositions $ I = (i_1, \ldots, i_{r-1}, i_r) $
and $J= (j_1, j_2, \ldots, j_s) $ define
\begin{align}
\label{def:multiplication-1}
& I \triangleright J =(i_1, \ldots, i_{r-1}, i_r + j_1, j_2, \ldots, j_s), &  \text{ with }
\ell( I \triangleright J ) = \ell(I) + \ell(J) - 1
& \intertext{and}
\label{def:multiplication-2}
&  I \cdot J = (i_1, \ldots, i_r , j_1,  \ldots, j_s), & \text{ with } \ell(I \cdot J) = \ell(I) + \ell(J)
\end{align}
Parts of the composition $ \widetilde{I} $ \textbf{conjugate} to a composition $ I $ can be read from the
diagram of the composition $ I $  from left to right and from bottom to top:
\begin{example}
For instance, if $ I = (3,1,1,4,2) $, then $ \widetilde{I} = (1,2,1,1,4,1,1)$

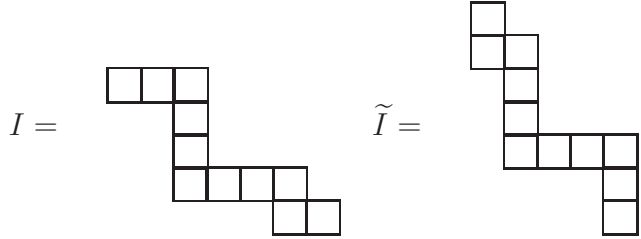
\centerline{
\setlength{\unitlength}{0.25pt}
\begin{picture}(900,400)
\put(-150,150){\mbox{$I=$}}
\put(0,200){\framebox(50,50){}}
\put(50,200){\framebox(50,50){}}
\put(100,200){\framebox(50,50){}}
\put(100,150){\framebox(50,50){}}
\put(100,100){\framebox(50,50){}}
\put(100,50){\framebox(50,50){}}
\put(150,50){\framebox(50,50){}}
\put(200,50){\framebox(50,50){}}
\put(250,50){\framebox(50,50){}}
\put(250,0){\framebox(50,50){}}
\put(300,0){\framebox(50,50){}}
\put(400,150){\mbox{$\widetilde{I}=$}}
\put(550,300){\framebox(50,50){}}
\put(550,250){\framebox(50,50){}}
\put(600,250){\framebox(50,50){}}
\put(600,200){\framebox(50,50){}}
\put(600,150){\framebox(50,50){}}
\put(600,100){\framebox(50,50){}}
\put(650,100){\framebox(50,50){}}
\put(700,100){\framebox(50,50){}}
\put(750,100){\framebox(50,50){}}
\put(750,50){\framebox(50,50){}}
\put(750,0){\framebox(50,50){}}
\end{picture}
}
\end{example}
\textbf{Reverse refinement order} for compositions is defined as follows.
Let $ I = (i_1, \ldots, i_n),
J=(j_1, \ldots, j_s) $, $ |J|=|I| $.
Then $ J \preceq I $ if  every part of $J$ can be obtained from consecutive parts of $I$:
\[
J = (i_1 + \ldots + i_{p_1}, i_{p_1 + 1} + \ldots + i_{p_2}, \ldots, i_{p_{k-1}+1} + \ldots + i_{p_k},
\ldots, i_{p_s } + \ldots + i_n )
\]
for some nonnegative $ p_1, \ldots, p_s $. (The convention $ p_0 = 0 $ will be implied below.) \\
\begin{example}
For instance, $ ( 3,3,2) = (3, 1 +2, 2) \preceq (3,1,2,2) $.

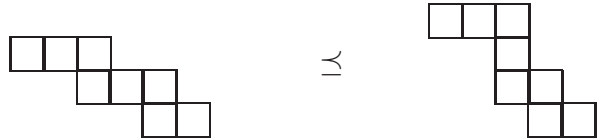
\centerline{
\setlength{\unitlength}{0.25pt}
\begin{picture}(300,300)
\put(0,150){\framebox(50,50){}}
\put(50,150){\framebox(50,50){}}
\put(100,150){\framebox(50,50){}}
\put(100,100){\framebox(50,50){}}
\put(150,100){\framebox(50,50){}}
\put(200,100){\framebox(50,50){}}
\put(200,50){\framebox(50,50){}}
\put(250,50){\framebox(50,50){}}
\end{picture}
\begin{picture}(300,300)
\put(150,150){\mbox{$\preceq $}}
\end{picture}
\begin{picture}(300,300)
\put(0,200){\framebox(50,50){}}
\put(50,200){\framebox(50,50){}}
\put(100,200){\framebox(50,50){}}
\put(100,150){\framebox(50,50){}}
\put(100,100){\framebox(50,50){}}
\put(150,100){\framebox(50,50){}}
\put(150,50){\framebox(50,50){}}
\put(200,50){\framebox(50,50){}}
\end{picture}
}
\end{example}
If $I(n)$ is the sequence of all compositions of $n$. Then,
$I(n) = ( 1 \cdot I(n-1) , 1 \triangleright I(n-1) ) $,
where $ 1 \cdot I(n-1) $ and $ 1 \triangleright I(n-1) $ denote respectively the compositions obtained from the compositions of $ I(n-1)$
 by adding $1 $ to their first part according to rules described in (\ref{def:multiplication-1}) and (\ref{def:multiplication-2}) while
other parts remain unchanged. \\
A \textbf{partition} is a composition with weakly decreasing parts, i.e.
\[
\la = ( \lambda_1, \ldots, \lambda_n) \mbox{ with } \lambda_1 \geq \lambda_2 \geq \ldots \geq \lambda_n
\]
The number of times an integer $ i $ occurs in a partition
$\la$ is denoted by $ \mathtt{m_i}(\la) $.
\subsection{Commutative Symmetric Functions.}
\label{sec-com}
Recall notations and definitions form \cite{M}. For every partition $ \la $ one defines monomial symmetric function $ m_{\la} $. \\
With that
 multiplicative bases -- power sums, elementary, and complete symmetric functions are introduced as follows:
\begin{align*}
& p_n = m_{(n)}
& p_{\la} = \prod_{i=1}^{\ell(\la)} p_{\lambda_i} \\
& e_n = m_{1^n}
& e_{\la} = \prod_{i=1}^{\ell(\la)} e_{\lambda_i} \\
& h_n = \sum_{|\la| = n} m_{\la}
& h_{\la} = \prod_{i=1}^{\ell(\la)} h_{\lambda_i} \\
\end{align*}
With involution $ \omega $ defined by
\[
\omega( p_{\la} ) = (-1)^{\ell(\la) - |\la|} p_{\la}
\]
one further has
\begin{align*}
& \omega( h_{\la}) = e_{\la} \\
& \omega( m_{\la} ) = f_{\la},
\end{align*}
where the second line defines forgotten symmetric functions.

Furthermore
\begin{equation}
\label{u}
u_{\boldsymbol{\mu}} = \prod_{ i \geq 1 } \mathtt{m_i} (\boldsymbol{\mu})
\end{equation}
Augmented monomial symmetric function $ \widetilde{m}_{\boldsymbol{\mu}} $ as in Exercise 10, \S 6, Ch. I of \cite{M}:
\begin{equation}
\label{augmon}
\widetilde{m}_{\boldsymbol{\mu}} = u_{\boldsymbol{\mu}} m_{\boldsymbol{\mu}}
\end{equation}
\subsection{Quasideterminants.}
To define noncommutative symmetric functions I will need the notion of a quasideterminant.
The definition of a quasideterminant for an arbitrary matrix was given in \cite{GL}. In general quasideterminant
is not polynomial in its entries. All the matrices that come up in what follows will be almost-triangular with constants
above the main diagonal. Moreover, in principle a matrix $ n \times n $ has $ n^2 $ quasideterminants, which can
be calculated with respect to any element of the matrix. In what follows by quasideterminant I will always mean
the quasideterminant with respect to the element in the lower left corner (see below).

Consider a quasideterminant of an $ n \times n $ almost triangular matrix with free entries $ a_{ij} $ and off-diagonal
elements coming from the first $ n -1 $ letters of an invertible alphabet $\mathbb{B}$.
Such quasideterminant  is polynomial in its entries and according to
Proposition 4.7 of \cite{GR2} can be written as:
\begin{align}
\label{def-quasidet}
& Q_n(\mathbb{B}) \equiv Q_n( b_1, \ldots, b_{n-1}) = \begin{vmatrix}
a_{11} & b_1 & 0 & \ldots & \ldots & \ldots & \ldots  \\
a_{21} & a_{22} & b_2 & 0 & \ldots & \ldots & \ldots  \\
\vdots & \vdots & \vdots & \vdots & \vdots & \vdots & \vdots  \\
a_{j1} & a_{j2} & \ldots  & a_{jj} & b_j  & 0 & \ldots  \\
\vdots & \vdots & \vdots & \vdots & \vdots & \vdots & \vdots  \\
a_{n-1 \ 1} & a_{n-1 \ 2} & \ldots & \ldots & \ldots  & \ldots & b_{n - 1} \\
\fbox{$a_{n1}$} & a_{n2} &  \ldots & a_{ n j} & \ldots & \ldots & a_{nn}
\end{vmatrix}= \nonumber \\
= & \sum_{n \geq j_1 > \ldots > j_k > 1} (-1)^{k+1} a_{ n j_1} b_{j_1 - 1}^{-1} a_{j_1 - 1 \ j_2}b_{j_2 -1}^{-1}
a_{j_2 - 1 \ j_3} \ldots b_{j_k -1}^{-1}a_{j_k - 1 \ 1}
\end{align}
From now on I will assume that $ b_j$'s commute among themselves and with all $ a_{ij} $.
\subsection{Noncommutative Symmetric Functions.}
As mentioned before, the algebra of noncommutative symmetric functions $ \mathbf{NSym} $ has been first defined by Gelfand et. al \cite{GL},
where noncommutative elementary symmetric functions were taken as generators of the algebra. For the purposes of this paper it will be more
 convenient to take power sums $ \Psi_n $ (power sums of the first kind in the terminology of \cite{GL}). \\
Given a (possible infinite) set of noncommutative variables $z_i$, define a noncommutative
\textbf{power sum} symmetric function
\begin{equation}
\label{def-power sums}
\Psi_n = \sum_{i=1} z_i^n
\end{equation}
Then it has been shown in \cite{GL} that noncommutative elementary symmetric functions as a  quasideterminant of an $ n $ by $ n $ matrix:
\begin{equation}
\label{def-elementary}
 n \Lambda_n = (-1)^{n - 1}
\begin{vmatrix}
\Psi_{1} & 1 & 0 & \dots & 0& 0\\
\Psi_{ 2} & \Psi_{1} & 2 & \ldots & 0 & 0 \\
\vdots &  \vdots & \vdots & \vdots &  \vdots & \vdots \\
\Psi_{n-1} & \ldots & \ldots & \ldots&  \Psi_{1}& n - 1 \\
\fbox{$\Psi_{n}$} & \ldots & \ldots & \ldots &\Psi_{2}& \Psi_{1}
\end{vmatrix}
\end{equation}
and noncommutative complete (homogeneous) symmetric functions
\begin{equation}
\label{def-complete}
 n S_n =
\begin{vmatrix}
\Psi_{1} & -( n - 1) & 0 & \dots & 0& 0\\
\Psi_{ 2} & \Psi_{1} & -( n - 2) & \ldots & 0 & 0 \\
\vdots &  \vdots & \vdots & \vdots &  \vdots & \vdots \\
\Psi_{n-1} & \ldots & \ldots & \ldots&  \Psi_{1}& - 1 \\
\fbox{$\Psi_{n}$} & \ldots & \ldots & \ldots &\Psi_{2}& \Psi_{1}
\end{vmatrix} \\
\end{equation}
For every composition $ I = ( i_1, \ldots, i_n)$ define multiplicative bases:
\begin{align}
& \text{ power sums } & \Psi^I = \Psi_{i_1} \Psi_{i_2} \dots \Psi_{i_n} \\
& \text{ complete symmetric functions } & S^I = S_{i_1} S_{i_2} \ldots S_{i_n} \\
& \text{ elementary symmetric functions } & \Lambda^I = \Lambda_{i_1} \Lambda_{i_2} \ldots \Lambda_{i_n}
\end{align}
Ribbon Schur functions have been defined in \cite{GL, G} as
\begin{equation}
\label{def:ribbon-schur}
R^I = (-1)^{k-1}
\begin{vmatrix}
S_{i_k} & 1 & 0 & 0  & \dots & \\
\vdots & & &  & \\
S_{i_3 + \dots + i_k} &  \dots & S_{i_3} & 1 & 0 \\
S_{i_2 + \dots + i_k} & \ldots & S_{i_2 + i_3} & S_{i_2} & 1 \\
\boxed{S_{i_1+ \ldots + i_k}} & \ldots & S_{i_1+ i_2 + i_3} & S_{i_1+ i_2} & S_{i_1}
\end{vmatrix} = \sum_{J \preceq I}
(-1)^{\ell(J) - \ell(I)} S^J ,
\end{equation}
Furthermore same authors have shown that classical involution can be lifted to the noncommutative setting:
\begin{align}
\label{eq-invol}
& \text{ if \ \ } \omega(\Psi_k ) = (-1)^{k-1} \Psi_k \\
& \text{ then \ \ } \omega( S_k ) = \Lambda_k \notag \\
& \text{as well as \ \ } \omega( R^I) = R^{I^{\widetilde{}}} \notag
\end{align}
\section{Results.}
\label{sec-results}
\begin{enumerate}
\item{\textit{Noncommutative Monomial, Forgotten and Fundamental Symmetric Functions.}}

Upon the introduction in Section \ref{sec-prim} of two linear bases in \textsf{\textbf{NSym}}... The first basis is that of
\textbf{noncommutative monomial} symmetric functions, denoted by $ M^I $, which are noncommutative analogs of
 monomial symmetric functions $ m_{\la} $ in the classical theory.

The two functions enjoy the following relationship,
\[
\widetilde{m}_{\boldsymbol{\mu}} =  \sum_{\mathfrak{S}_n} M^{I},
\]
where the sum is over all permutations of composition $ I $ and
$ \boldsymbol{\mu} $ is the partition obtained by ordering parts of $I$.

This relationship will be proved in the Section \ref{sec-monomial}.

The second basis is that of \textbf{noncommutative forgotten} symmetric functions denoted by $ F^{I} $.

It follows that complete symmetric functions continue to be worthy of their
name just as in the commutative case where $ h_n $ is a sum of all monomial functions of the same degree.
The noncommutative version of this relationship is
\[
S_n = \sum_{|I| = n } M^I
\]
In fact one can make a stronger statement
\[
F^{I} = \sum_{J \preceq I} M^J,
\]
where the sum is over all compositions that are less fine than $ I $.
This statement will be proved in Section \ref{sec-triang}.\\
Finally, by analogy with Gessel's   theory of quasi-symmetric functions \cite{G} I define and study
\textbf{fundamental} noncommutative symmetric functions $ L^I $ in Section \ref{sec-fundamental}.
\item{\textit{Ribbon Schur Functions are Fundamental Positive.}}
It turns out that the expansion of ribbon Schur functions is non-negative both monomial and fundamental bases, see Section \ref{sec-transitions}.
\item{\textit{Noncommutative Cauchy Identity and an Analog of Hall scalar product.}}
There is a noncommutative analog of Cauchy identity, with both ribbon Schur functions and noncommutative fundamental symmetric functions playing
a role similar to that of Schur functions in the commutative case:
\[
\sum_I M^I S^I = \sum_I L^I R^I
\]
proved in Section \ref{sec-cauchy}, which allows one to introduce a noncommutative analog of Hall scalar product:
\[
\langle M^I \mid S^J \rangle = \delta_{IJ}
\]
This scalar product shares some properties with its classical counterpart and they are studies further in the same section.
\end{enumerate}
\section{Noncommutative Theory.}
\label{sec-noncom}
\subsection{Noncommutative Monomial and Forgotten Symmetric Functions.}
\label{sec-prim}
Let me begin by introducing one of the main objects of interest -- \textbf{noncommutative monomial symmetric function}
corresponding to a composition $ I = (i_1, \ldots, i_n)$. It is defined as a  quasideterminant of the following $ n $ by $ n $ matrix:
\begin{definition}
\begin{equation}
\label{def-primm}
 n M^I \equiv n M^{(i_1, \ldots, i_n)} = (-1)^{n - 1}
\begin{vmatrix}
\Psi_{i_n} & 1 & 0 & \dots & 0& 0\\
\Psi_{ i_{n -1} + i_n} & \Psi_{i_{n - 1}} & 2 & \ldots & 0 & 0 \\
\vdots &  \vdots & \vdots & \vdots &  \vdots & \vdots \\
\Psi_{i_2 + \ldots + i_n} & \ldots & \ldots & \ldots&  \Psi_{i_2}& n - 1 \\
\fbox{$\Psi_{i_1 + \ldots + i_n}$} & \ldots & \ldots & \ldots &\Psi_{i_1 + i_2}& \Psi_{i_1}
\end{vmatrix}
\end{equation}
\end{definition}
where $n$ is the length of $ I $.
In particular
\[
\Lambda_n= M^{1^n}
\]
where $ \Lambda_n $ is an elementary symmetric function (\ref{def-elementary}). \\
Equivalently, if $ J = (j_1, \ldots, j_s) \preceq I =(i_1, \ldots, i_n) $, i.e.
\[
J = (i_1 + \ldots + i_{p_1}, i_{p_1 + 1} + \ldots + i_{p_2}, \ldots, i_{p_{k-1}+1} + \ldots + i_{p_k},
\ldots, i_{p_s } + \ldots + i_n )
\]
for some nonnegative $ p_1, \ldots, p_s $ ($ p_0 = 0 $), then
\begin{state}
\begin{equation}
\label{eq:pi-through-Psi}
 M^I = \sum_{J \preceq I } \frac{(-1)^{\ell(I) - \ell(J)}}{ \prod_{k=0}^{s-1} (\ell(I) - p_k)} \Psi^J,
\end{equation}
where $ s = \ell(J)$.
\end{state}
This expression follows from the definition of a quasideterminant (\ref{def-quasidet}).
\begin{example}
\[
M^{312} = \frac1{3}
\begin{vmatrix}
\Psi_2 & 1 & 0 \\
 \Psi_3 & \Psi_1 & 2 \\
\boxed{ \Psi_6 } & \Psi_4 & \Psi_3
\end{vmatrix} = \frac1{3} \left( \Psi_6 - \Psi_4 \Psi_2 - \frac1{2} \Psi_3^2 + \frac1{2} \Psi_3 \Psi_1 \Psi_2 \right)
\]
\end{example}
Also define \textbf{noncommutative forgotten symmetric function} corresponding to a composition
$ I = (i_1, \ldots, i_n)$ as an $ n $ by $ n $ quasideterminant:
\begin{definition}
\begin{equation}
\label{def-primf}
 n F^I \equiv n F^{(i_1, \ldots, i_n)} =
\begin{vmatrix}
\Psi_{i_n} & -( n - 1) & 0 & \dots & 0& 0\\
\Psi_{ i_{n -1} + i_n} & \Psi_{i_{n - 1}} & -( n - 2) & \ldots & 0 & 0 \\
\vdots &  \vdots & \vdots & \vdots &  \vdots & \vdots \\
\Psi_{i_2 + \ldots + i_n} & \ldots & \ldots & \ldots&  \Psi_{i_2}& - 1 \\
\fbox{$\Psi_{i_1 + \ldots + i_n}$} & \ldots & \ldots & \ldots &\Psi_{i_1 + i_2}& \Psi_{i_1}
\end{vmatrix} \\
\end{equation}
\end{definition}
In particular
\[
F^{1^n} = S_n
\]
where $ S_n $ a homogeneous symmetric function (\ref{def-complete}). \\
Equivalently,
\begin{state}
\begin{equation}
\label{eq:pi-through-Psi}
 F^I = \sum_{J \preceq I } \frac1{ \prod_{k=1}^{s} p_k} \Psi^J,
\end{equation}
where $ s = \ell(J)$.
\end{state}
\begin{example}
\[
F^{213} = \frac1{3}
\begin{vmatrix}
\Psi_3 & -2 & 0 \\
\Psi_4 & \Psi_1 & -1 \\
\boxed{ \Psi_6} & \Psi_3 & \Psi_2
\end{vmatrix} = \frac1{3} \left( \Psi_6 + \frac1{2}\Psi_3^2 + \Psi_2 \Psi_4 + \frac1{2} \Psi_2 \Psi_1 \Psi_3 \right)
\]
\end{example}
As noted above (see (\ref{def-quasidet}))
it follows from Proposition 4.7 of \cite{GR2} that all $ M^I $ and $ F^I $ are polynomial in power sums.

Notice that if involution $ \omega $ as defined in (\ref{eq-invol}) is extended
(compare Prop. 3.9 in \cite{GL})
\[
\omega(\Psi^I) = (-1)^{|I|-\ell(I)} \Psi^{\bar{I}}
\]
then
\begin{equation}
\label{duality-pi}
\omega(M^I) = (-1)^{|I| - \ell(I)} F^{\bar{I}},
\end{equation}
where $ \bar{I} $ is the reversed composition.
\subsection{Recursion Relations, Pieri formulas, and Multiplicative Structure.}
The definition of the primitive monomial functions
(\ref{def-primm}) implies a linear relationship between primitive monomial functions
obtained from the same composition due to Theorem 1.8 of \cite{GR1}.

This relationship can be stated as follows (and may be considered a generalization of formulas (31) in
Proposition 3.3 in \cite{GL}):
\begin{state}{Newton-type relations.}
\begin{align}
\label{eq-linear}
& n M^{i_1, \ldots, i_n} = \Psi_{i_1}M^{i_2, \ldots , i_n} - \Psi_{i_1 + i_2}M^{i_3, \ldots, i_n}
+ \ldots + (-1)^{s - 1} \Psi_{i_1 + \ldots + i_s} M^{i_{s +1},  \ldots, i_n} + \ldots +
(-1)^{n -1} \Psi_{i_1 + \ldots + i_n} \\
\label{eq-linear:forgotten}
& n F^{i_1, \ldots, i_n} = F^{i_1, \ldots , i_{n-1}} \Psi_{i_n} + F^{i_1, \ldots, i_{n-2}}\Psi_{i_{n-1} + i_n}
+ \ldots + F^{i_1 + \ldots + i_s} \Psi_{i_{s +1},  \ldots, i_n} + \ldots +
 \Psi_{i_1 + \ldots + i_n}
\end{align}
\end{state}
Moreover, Newton-type relationships above imply rules of multiplication by a power sum  of an arbitrary primitive
monomial function on the left and forgotten on the right:
\begin{lemma}{Pieri-like formula for primitive monomials and forgotten.}

Let $ I $ be a composition with $ \ell(I) = n $, then
\begin{align}
\label{eq-pieri}
& \Psi_r \cdot M^{I} = (n + 1)M^{ (r) \cdot I} + n \ M^{ (r) \triangleright I} \\
\label{eq-pieri:forgotten}
& F^{I} \cdot \Psi_r = (n+1) F^{ I \cdot (r)} - n \ F^{ I \triangleright (r)}
\end{align}
\end{lemma}
\begin{proof}
Writing (\ref{eq-linear}) for two compositions $ (r) \cdot I \equiv (r, i_1, \ldots, i_n)$ and
$ (r) \triangleright I \equiv (r + i_1, i_2, \ldots, i_n) $ (of lengths $ n + 1$ and $ n$ correspondingly) one gets:
\begin{align*}
&  (n + 1)M^{(r) \cdot I} = \Psi_{r}M^{(i_1, \ldots , i_n)} - \Psi_{r + i_1}M^{(i_2, \ldots, i_n)}
 + \ldots + (-1)^{n + 1} \Psi_{r + i_1 + \ldots + i_n} \\
& n \ M^{(r) \triangleright I} =  \Psi_{r + i_1}M^{(i_2, \ldots, i_n)}
+  \ldots + (-1)^{n} \Psi_{r +i_1 + \ldots + i_n}
\end{align*}
Formula (\ref{eq-pieri}) follows by adding these expressions. \\
Equation (\ref{eq-pieri:forgotten}) can be seen as a result of application of involution to
(\ref{eq-pieri}) or proved exactly the same way as (\ref{eq-pieri}).
\end{proof}
More generally, the product of two monomial symmetric functions has the following expansion:
\begin{state}
\begin{equation}
\label{eq-pimultiply}
M^{I} \cdot M^{J} = \sum_{K \preceq I} \binom{\ell(K) + \ell(J) }{\ell(I)} M^{ K \cdot J} +
\binom{\ell(K) + \ell(J) - 1}{\ell(I)} M^{ K \triangleright J},
\end{equation}
where the sum is over all compositions preceding $ J $ in the reverse refinement order.
\end{state}
\begin{proof}
The proof of equation (\ref{eq-pimultiply}) proceeds by induction on the number of parts of the composition $ J $, Lemma (\ref{eq-pieri})
being the initialization and uses Lemma 4.1 in \cite{GL}.
\end{proof}
\begin{state}
\begin{equation}
\label{eq-phimultiply}
F^{I} \cdot F^{J} = \sum_{ K \preceq J } (-1)^{\ell(K) - \ell(J) + 1} \binom{ \ell(I) + \ell(K)}{\ell(J)} F^{I \cdot K} +
\sum_{ K \preceq J } (-1)^{\ell(K) - \ell(J)} \binom{ \ell(I) + \ell(K) - 1}{\ell(J)} F^{I \triangleright K}
\end{equation}
\end{state}
\begin{proof}
The formula follows by applying $ \omega $ to the formula (\ref{eq-pimultiply}).
\end{proof}
\subsection{Fundamental Noncommutative Symmetric Functions.}
\label{sec-fundamental}
In this section I define the noncommutative analog of Gessel's \cite{G} fundamental symmetric function.
More precisely,  a \textbf{fundamental} noncommutative symmetric function is
\begin{definition}
\begin{equation}
\label{def-fundamental}
L^I = \sum_{J \succeq I} M^J
\end{equation}
\end{definition}
Conversely
\begin{equation}
\label{eq:monomial-through-fundamental}
M^I = \sum_{J \succeq I} (-1)^{\ell(J) - \ell(I)} L^J
\end{equation}
Just like its quasi-symmetric prototype, this function behaves nicely under the involution $ \omega $:
\begin{state}
\label{prop-duality-fundamental}
\[
\omega \left( L^I \right) =  L^{\widetilde{I}}
\]
\end{state}
The proof of this fact will be given in Section \ref{sec-cauchy}.
\subsection{Multiplication of Fundamental Noncommutative Symmetric Functions.}
It is natural to inquire about multiplication rule for noncommutative fundamental symmetric functions. \\
In fact their product is non-negative (in the
same basis) and given by (I would like to thank to Jean-Christophe Novelli for pointing out this formula.):
\begin{state}
\[
L^I \cdot L^J = \sum_{K \preceq I, \ M \succeq J} \binom{ |I| + \ell(J) - \ell(I)}{\ell(K) + \ell(M) - \ell(I)} L^{ K \cdot M}
+ \binom{ |I| + \ell(J) - \ell(I)}{\ell(K) + \ell(M) -1 - \ell(I)} L^{ K \triangleright M}
\]
\end{state}
In principle, the proof relies on the direct application of the multiplication rule of monomials. \\
It will be important to notice that if $ \{S\} \succeq K $ (and $ \{T\} \succeq V $)
is a set of compositions no less then $ K $ ($V$ respectively), then all compositions no less then $ K \cdot V $ are of the
form $ S \cdot T $; and all compositions no less then $ K \triangleright V $ are either of the
form $ S \cdot T $ or $ S \triangleright T $. That is
\begin{align}
\label{eq:larger-compositions}
& \{K \cdot V\} = \{ S \} \cdot \{ T \} \text{ and } \notag \\
& \{K \triangleright V \} = \{ S \} \triangleright \{T \} \cup \{ S \} \cdot \{T \},
\end{align}
here both multiplications are understood as element-wise operations. \\
In addition two identities will also prove useful,
\begin{lemma}
\begin{align}
\label{eq:identity1}
& \sum_{J: \ J \succeq I} \binom{X}{Y + \ell(J)} = \binom{X+ |I| - \ell(I)}{Y + |I| } \\
\label{eq:identity2}
& \sum_{M: \ S \succeq M \succeq J }(-1)^{ \ell(S) - \ell(M)} \binom{ X + \ell(M) }{Y} = \binom{X + \ell(J)}{Y - \ell(S) + \ell(J)}
\end{align}
\end{lemma}
\begin{proof}
I will break up the calculation in two parts. First I will calculate $ L^K \cdot M^V $ and then sum over all $ V \succeq J $.
\begin{align*}
&  L^I \cdot M^V = \sum_{W \succeq I} M^W \cdot M^V = \sum_{W \succeq I} \left( \sum_{ K \preceq W } \binom{ \ell(K) + \ell(V) }{\ell(W)}
M^{ K \cdot V } + \binom{ \ell(K) + \ell(V) -1}{\ell(W)} M^{ K \triangleright V } \right) = \\
& \stackrel{\text{ by } (\ref{eq:monomial-through-fundamental}) \text{ and } (\ref{eq:larger-compositions})}{=}
 \sum_{W \succeq I} \sum_{ K \preceq W } \binom{ \ell(K) + \ell(V) }{\ell(W)} \left(
\sum_{ S \succeq K, \ T \succeq V} (-1)^{\ell(S) + \ell(T) - \ell(K) - \ell(V)} L^{ S \cdot T} \right) + \\
& +  \binom{ \ell(K) + \ell(V) -1}{\ell(W)} \left(
\sum_{ S \succeq K, \ T \succeq V} (-1)^{\ell(S) + \ell(T) - \ell(K) - \ell(V)} \left( L^{ S \triangleright T} - L^{ S \cdot T} \right)\right) = \\
& =  \sum_{W \succeq I} \sum_{ K \preceq W } \binom{ \ell(K) + \ell(V)-1 }{\ell(W) -1} \left(
\sum_{ S \succeq K, \ T \succeq V} (-1)^{\ell(S) + \ell(T) - \ell(K) - \ell(V)} L^{ S \cdot T} \right) + \\
& +  \binom{ \ell(K) + \ell(V) -1}{\ell(W)} \left(
\sum_{ S \succeq K, \ T \succeq V} (-1)^{\ell(S) + \ell(T) - \ell(K) - \ell(V)} L^{ S \triangleright T} \right) = \\
& =  \sum_{S \preceq I, \ T \succeq V} (-1)^{ \ell(T)  - \ell(V)} \sum_{W  \succeq I } \sum_{ K \preceq S} (-1)^{\ell(S)  - \ell(K) }
\left[ \binom{\ell(K) + \ell(V) -1}{ \ell(W) -1} L^{ S \cdot T} +
\binom{\ell(K) + \ell(V) -1}{ \ell(W)} L^{ S \triangleright T} \right] \stackrel{\text{ by } (\ref{eq:identity1})}{=} \\
& = \sum_{S \preceq I, \ T \succeq V} (-1)^{ \ell(T)  - \ell(V)} \sum_{W  \succeq I }  \binom{ \ell(V) -1 + 1}{\ell(W) - \ell(S)}  L^{ S \cdot T}+
 \binom{ \ell(V) -1 + 1}{\ell(W) - \ell(S) +1} L^{ S \triangleright T} \stackrel{\text{ by } (\ref{eq:identity2})}{=} \\
& =  \sum_{S \preceq I, \ T \succeq V} (-1)^{ \ell(T)  - \ell(V)} \left[
\binom{ \ell(V) + |I| - \ell(I)}{|I| -\ell(S)}  L^{ S \cdot T}+ \binom{ \ell(V) + |I| - \ell(I)}{|I| - \ell(S) + 1}  L^{ S \triangleright T}
\right]
\end{align*}
Finally,
\begin{align*}
& L^I \cdot L^J = L^I \sum_{V \succeq J} M^V = \sum_{V \succeq J}
\sum_{S \preceq I, \ T \succeq V} (-1)^{ \ell(T)  - \ell(V)} \left[
\binom{ \ell(V) + |I| - \ell(I)}{|I| - \ell(S)}  L^{ S \cdot T}+ \binom{ \ell(V) + |I| - \ell(I)}{|I| - \ell(S) + 1}  L^{ S \triangleright T}
\right]= \\
& = \sum_{S \preceq I, \ T \succeq J} \sum_{ T \succeq V \succeq J}
(-1)^{ \ell(T)  - \ell(V)} \left[
\binom{ \ell(V) + |I| - \ell(I)}{|I| - \ell(S)}  L^{ S \cdot T}+ \binom{ \ell(V) + |I| - \ell(I)}{|I| - \ell(S) + 1}  L^{ S \triangleright T}
\right] \stackrel{\text{ by } (\ref{eq:identity2})}{=} \\
& = \sum_{S \preceq I, \ T \succeq J}
\binom{  |I| - \ell(I) + \ell(J) }{|I| - \ell(S) - \ell(T) + \ell(J)}  L^{ S \cdot T}+ \binom{  |I| - \ell(I) + \ell(J)}{|I| - \ell(S) + 1
- \ell(T) + \ell(J)} L^{ S \triangleright T} = \\
& = \sum_{S \preceq I, \ T \succeq J}
\binom{  |I| - \ell(I) + \ell(J) }{\ell(S) + \ell(T) - \ell(I)}  L^{ S \cdot T}+ \binom{  |I| - \ell(I) + \ell(J)}{ \ell(S)
+ \ell(T) - \ell(I) -1} L^{ S \triangleright T}
\qedhere
\end{align*}
\end{proof}
\newpage
\section{Transitions between Different Bases.}
\label{sec-transitions}
\subsection{Identity for Quasi-Determinants of Almost-Triangular Matrices and Transition Matrices between Forgotten and Monomial Bases.}
\label{sec-triang}
The property of triangularity follows from a general identity for quasideterminants to be
presented in a separate paper \cite{LT}.

Let an operator $ T_j $  act on $ Q_n (\mathbb{B}) $, see (\ref{def-quasidet}), by simultaneously removing
 $(j +1)^{\text{th}}$th column and $ j^{\text{th}}$ row
(the column and row that intersect at the off-diagonal element $ b_j $).
Fill the resulting $ (n -1) \times (n-1) $ matrix is filled in with the first $ ( n - 2) $ letters
of the alphabet $ \mathbb{B} $.
Further, for a sequence of distinct positive integers $ J =(j_1,j_2, \ldots, j_k) \subset \mathbb{N} $
define
\[
T_J = \prod_{s=1}^k T_{j_s}, \text{ and } \ell(J) = k
\]
Take $ \mathbb{B} = \mathbb{N} $. Then the following identity is true \cite{LT}.
\begin{state}
\label{pr-kaleid}
\begin{equation}
\label{eq-kaleid}
\frac1{n} Q_n(-(n-1), -(n-2), \ldots, -1)  =
\sum_{J}\frac{(-1)^{ n - k - 1}}{n - k} T_J Q_n(1,2,\ldots, n-1),
\end{equation}
where the sum is over all subsets $ J  \subseteq [ 1, 2, \dots, n - 1 ] $.
\end{state}
\begin{example}
\label{ex-4kaleid}
Consider a four by four quasideterminant $ Q_4(-3,-2,-1) $ and its kaleidoscopic expansion:
\begin{align*}
& \frac1{4} \begin{vmatrix}
a_{11} & -3 & 0 & 0 \\
a_{21} & a_{22} & -2 & 0 \\
a_{31} & a_{32} & a_{33} & -1 \\
\fbox{$a_{41}$} & a_{42} & a_{43} & a_{44}
\end{vmatrix} = \\
 = & \left( - \frac1{4} T_{\emptyset} + \frac1{3} \left( T_1 + T_2 + T_3 \right) -
\frac1{2} \left( T_1 T_2 + T_1 T_3 + T_2 T_3 \right) + T_1 T_2 T_3 T_4 \right) Q_4(1,2,3) = \\
 = & -\frac1{4}\begin{vmatrix}
a_{11} & 1 & 0 & 0 \\
a_{21} & a_{22} & 2 & 0 \\
a_{31} & a_{32} & a_{33} & 3 \\
\fbox{$a_{41}$} & a_{42} & a_{43} & a_{44}
\end{vmatrix}
+ \frac1{3}\begin{vmatrix}
a_{21}  & 1 & 0 \\
a_{31}  & a_{33} & 2 \\
\fbox{$a_{41}$}  & a_{43} & a_{44}
\end{vmatrix} +
\frac1{3}\begin{vmatrix}
a_{11} & 1  & 0 \\
a_{31} & a_{32}  & 2 \\
\fbox{$a_{41}$} & a_{42}  & a_{44}
\end{vmatrix} +
\frac1{3}\begin{vmatrix}
a_{11} & 1 & 0  \\
a_{21} & a_{22}  & 2 \\
\fbox{$a_{41}$} & a_{42} & a_{43}
\end{vmatrix} - \\
& - \frac1{2}\begin{vmatrix}
a_{31} & 1  \\
\fbox{$a_{41}$} & a_{44}
\end{vmatrix} - \frac1{2}\begin{vmatrix}
a_{12} & 1 &  \\
\fbox{$a_{41}$} & a_{43}
\end{vmatrix} -
\frac1{2}\begin{vmatrix}
a_{11} & 1 \\
\fbox{$a_{41}$} & a_{42}
\end{vmatrix}  + a_{41}
\end{align*}
\end{example}

If $ a_{kj} = \Psi_{i_{n - k +1} + \ldots + i_{n - j +1}} $, where $ ( i_1, \ldots, i_n) $ are parts of the composition $ I $
 then (\ref{eq-kaleid}) implies
\begin{state}
\begin{equation}
\label{eq-triang}
F^{I} = \sum_{J \preceq I} M^J,
\end{equation}
where the sum is over compositions in the reverse refinement order.

And conversely, by the inclusion-exclusion principle,
\begin{equation}
\label{eq-converse-triang}
M^{I} = \sum_{J \preceq I} (-1)^{\ell(I) - \ell(J)} F^J
\end{equation}
\end{state}
\begin{example}
Continuing Example \ref{ex-4kaleid}, consider an expansion of $ F^{2,2,1,3} $. Therefore
let $ I = (2,2,1,3) $ and $ a_{kj} = \Psi_{i_{n - k +1} + \ldots + i_{n - j +1}}  $, then
\begin{align*}
& \frac1{4} \begin{vmatrix}
\Psi_3 & -3 & 0 & 0  \\
 \Psi_4 & \Psi_1 & -2 & 0\\
\Psi_6 & \Psi_3 & \Psi_2 & -1 \\
 \fbox{$\Psi_8$}& \Psi_5 & \Psi_4 & \Psi_2 \\
 \end{vmatrix} =
 - \frac1{4} \begin{vmatrix}
\Psi_3 & 1 & 0 & 0  \\
 \Psi_4 & \Psi_1 & 2 & 0\\
\Psi_6 & \Psi_3 & \Psi_2 & 3 \\
 \fbox{$\Psi_8$}& \Psi_5 & \Psi_4 & \Psi_2 \\
 \end{vmatrix} + \frac1{3}
 \begin{vmatrix}
\Psi_4  & 1 & 0 \\
\Psi_6  & \Psi_2 & 2 \\
\fbox{$\Psi_8$}  &\Psi_4 & \Psi_2
\end{vmatrix} +
\frac1{3}\begin{vmatrix}
\Psi_3 & 1  & 0 \\
\Psi_6 & \Psi_3  & 2 \\
\fbox{$\Psi_8$}  &\Psi_5 & \Psi_2
\end{vmatrix} + \\
& +\frac1{3}\begin{vmatrix}
\Psi_3 & 1 & 0  \\
\Psi_4 & \Psi_1  & 2 \\
\fbox{$\Psi_8$}  &\Psi_5 & \Psi_4
\end{vmatrix}
- \frac1{2}\begin{vmatrix}
\Psi_6 & 1  \\
\fbox{$\Psi_8$} & \Psi_2
\end{vmatrix} - \frac1{2}\begin{vmatrix}
\Psi_4 & 1 &  \\
\fbox{$\Psi_8$} & \Psi_4
\end{vmatrix} -
\frac1{2}\begin{vmatrix}
\Psi_3 & 1 \\
\fbox{$\Psi_8$} & \Psi_5
\end{vmatrix}  + \Psi_8,
\end{align*}
i.e.
\[
F^{2,2,1,3} = M^{2,2,1,3} + M^{2,2,4} + M^{2,3,3} + M^{4,1,3} + M^{2,6} + M^{4,4} + M^{5,3} + M^8
\]
\end{example}
\subsection{Transitions Between Complete and Power Sums.}
First recall the notation and transition formulas between power sums and complete symmetric functions obtained in \cite{GL}.
Let $I = (i_1, \ldots , i_m)$ be a composition. Define $\pi(I)$ as follows
\[
\pi_u(I) = i_1 (i_1 + i_2) \ldots (i_1 + i_2 + \ldots + i_m)
\]
In other words, $\pi_u(I) $ is the product of the successive partial sums of the entries of
the composition $I$. \\
Let  $J$ be a composition which is finer than $I$ and
$J = (J_1, . . . , J_m) $ be the unique decomposition of $J$ into compositions $(J_i)_{i=1,m} $ such that
$ |J_p| = i_p, p = 1,\ldots ,m $. Define
\[
\pi_u(J, I) = \prod_{i=1} \pi_u(J_i)
\]
Then, according to the Proposition 4.5 of \cite{GL}
\begin{align}
\label{eq:S-through-Psi}
& S^I = \sum_{J \succeq I} \frac1{\pi_u(J, I)} \Psi^J \\
\label{eq:Psi-through-S}
& \Psi^I = \sum_{J \succeq I} (-1)^{ \ell(J) - \ell(I)} lp(J,I) S^J
\end{align}
\subsection{Noncommutative Monomial and Power Sums.}
Consider $ J = (j_1, \ldots, j_s) \preceq I =(i_1, \ldots, i_n) $, i.e.
\[
J = (i_1 + \ldots + i_{p_1}, i_{p_1 + 1} + \ldots + i_{p_2}, \ldots, i_{p_{k-1}+1} + \ldots + i_{p_k},
\ldots, i_{p_s } + \ldots + i_n )
\]
for some nonnegative $ p_1, \ldots, p_s $.
Set $ p_0 = 0 $.
Then the formula for transition from power sum to  the monomial basis can be written as
\begin{state}
\begin{equation}
\label{eq:Psi-through-primitive}
\Psi^I = \sum_{J \preceq I} \prod_{k=1}^{\ell(J)} ( \ell(J) - k + 1)^{p_k - p_{k -1}} M^J
\end{equation}
\end{state}
\begin{proof}
By induction, using (\ref{eq-pieri}).
\end{proof}
\subsection{Monomial and Complete.}
Combining the expansion of monomial functions in power sums and power sums in complete, i.e.
equations (\ref{eq:Psi-through-S}) and (\ref{eq:pi-through-Psi}) one can obtain:
\begin{align}
\label{eq:pi-through-S}
& M^I = \sum_{J \preceq I, \ K \succeq J} \frac{(-1)^{\ell(I) - \ell(K)}}{ \prod_{k=0}^{s-1}( \ell(I) - p_k)} lp(K,J) S^K
 \\
 & \text{ and conversely} \notag \\
\label{eq:S-through-pi}
& S^I = \sum_{ J \succeq I, \ K \preceq J} \frac{ \prod_{k=1}^{\ell(K)}(\ell(K) - k + 1)^{p_k - p_{k-1}}}{ \pi_u(J,I)} M^K
\end{align}
\subsection{Ribbon Schur Functions, Monomial Basis and Noncommutative Kostka Numbers.}
The expansion of monomials through ribbon Schur functions can be obtained by combining (\ref{eq:pi-through-Psi}) and the formula for
the expansion of power sums through ribbon Schur from Proposition 4.23 of \cite{GL}:
\[
\Psi^I = \sum_{|J|=|I|} psr(J,I) R^J
\]
that is
\begin{equation}
\label{eq:pi-through-ribbon}
 M^I = \sum_{J \preceq I } \frac{(-1)^{\ell(I) - \ell(J)}}{ \prod_{k=0}^{s-1} (\ell(I) - p_k)} \Psi^J =
\sum_{J \preceq I, |K|=|I| } \frac{(-1)^{\ell(I) - \ell(J)}}{ \prod_{k=0}^{s-1} (\ell(I) - p_k)} psr(K,J) R^K
\end{equation}
Consider an expansion of Ribbon Schur functions in primitive monomials
\begin{align*}
& R^I  = \sum_{ J} K_{IJ} M^J,
\end{align*}
where $ K_{IM} $ can be called the \textbf{noncommutative Kostka numbers} by analogy with the classical case. \\
The following was stated as a conjecture \cite{LT1} and  has since been proven in \cite{N}. \\
\begin{state}
Noncommutative Kostka numbers are \textbf{positive integers}.
\end{state}
For some simple shapes one can obtain explicit formulas:
\begin{state}
\begin{align}
\label{eq:Kostka-hook}
& R^{k 1^r} =  \binom{k + r -1}{r} \sum_{|I|=k} M^{I \cdot 1^r} \\
\label{eq:lower-hook-monomial}
& R^{1^r k} = \sum_{|J|=r} \sum_{|I|=k} \binom{ \ell(I) + \ell(J) -1}{ r} M^{J \cdot I} +
\binom{  \ell(I) + \ell(J) -2}{ r} M^{J \triangleright I}
\end{align}
\end{state}
\begin{proof}
Both formulas are proved by induction on $ k $. For $ k = 1 $ and arbitrary $ r $ both formulas are obvious.
For the first formula one makes use of the product
\[
S_k \cdot \Lambda_r = R^{k 1^r} + R^{k+1 1^{r - 1}}
\]
for the induction step, while for the second --
\[
\Lambda_r \cdot S_k = R^{1^r k } + R^{1^{r-1} k +1}
\]
\end{proof}
\subsection{Ribbon Schur Functions, Fundamental Basis and Kostka-Gessel Numbers.}
\begin{equation}
\label{eq-ribbon-fundamental}
R^I = \sum_{J} G_{IJ} L^J,
\end{equation}
where $ G_{IJ} $ may be called Kostka-Gessel numbers. \\
The following statement was also stated as a conjecture \cite{LT1} and has since been proved in \cite{N}:
\begin{state}
Kostka-Gessel numbers are non-negative integers.
\end{state}
\section{Quasi-Symmetric Analogs of Power Sum, Elementary, Complete, and Ribbon Schur Bases.}
\label{sec-quasi-symmetric}
Mapping noncommutative monomial symmetric functions to quasi-symmetric monomials, one can identify analogs of multiplicative basis in
\textbf{\textsf{QSym}} corresponding to each composition $ I $. \\
Using the expansion of power sums, elementary, and complete in the monomial basis (\ref{eq:Psi-through-primitive}), (\ref{eq:S-through-pi})
define
\begin{definition}
\textbf{Quasi-symmetric powers sums}
\[
 \Psi_I  = \sum_{J \preceq I} \prod_{k=1}^{\ell(J)} ( \ell(J) - k + 1)^{p_k - p_{k -1}} M_J
 \]
\textbf{Quasi-symmetric complete}
\[
S_I = \sum_{ J \succeq I, \ K \preceq J} \frac{ \prod_{k=1}^{\ell(K)}(\ell(K) - k + 1)^{p_k - p_{k-1}}}{ \pi_u(J,I)} M_K
\]
and \textbf{quasi-symmetric elementary} $ \Lambda_I $ by applying $ \omega $ to the expression for $ S_I $. \\
Furthermore, one can define \textbf{quasi-symmetric ribbon Schur functions} through its expansion in the fundamental (or monomial) basis
\[
R_I = \sum_{J} G_{IJ} L_J,
\]
where $ G_{IJ}$ are Kostka-Gessel numbers introduced in (\ref{eq-ribbon-fundamental}) and $ L_J $ are Gessel's fundamental
quasi-symmetric functions.
\end{definition}
\newpage
\section{Noncommutative Cauchy Identity and a Scalar Product in \textbf{\textsf{NSym}}.}
\label{sec-cauchy}
Equipped with noncommutative monomial and fundamental symmetric functions, everything is ready for
the noncommutative version of the Cauchy identity. \\
\begin{state}
Let $ X $ and $ Y $ be noncommutative alphabets, then
\[
\sum_I M^I(X) S^I(Y) =  \sum_J L^J (X) R^J(Y) .
\]
\end{state}
\begin{proof}
\[
\sum_I M^I(X) S^I(Y) = \sum_I M^I (X) \left( \sum_{J \preceq I} R^J (Y) \right) = \sum_J \left( \sum_{I \succeq J}
M^I(X) \right) R^J (Y) = \sum_J L^J (X) R^J(Y) .
\]
\end{proof}
In fact, using the expansion of monomials and complete in power sums, one further has
\begin{align*}
& \sum_I M^I(X) S^I(Y)  =  \sum_{ I, \ K \succeq I \succeq J}
\frac{(-1)^{\ell(I) - \ell(J)}}{ \prod_{k=0}^{\ell(J)-1} (\ell(I) - p_k) \pi_u(K, I)} \Psi^J (X) \Psi^K (Y)
\end{align*}

Using the noncommutative Cauchy identity as a starting point,
define a scalar product by declaring monomial symmetric functions be dual to complete:
\begin{definition}
\label{eq:pairing}
\[
\langle M^I \mid S^J \rangle = \delta_{I J},
\]
\end{definition}
where I have introduced a notation:
\begin{align*}
& \delta_{IJ} = \begin{cases}
1, \quad I = J \\
0, \quad \text{otherwise}
\end{cases}
\end{align*}
I will take this opportunity to introduce another notation:
\begin{align*}
& \theta(I - J) = \begin{cases}
1, \quad I \succeq J \\
0, \quad \text{otherwise}
\end{cases}
\end{align*}
Then fundamental symmetric functions turn out to be dual to ribbon Schur:
\begin{state}
\label{prop:duality-fundamental-ribbon}
\[
\langle L^I \mid R^J \rangle = \delta_{I J}
\]
\end{state}
Before proving this statement, it is useful to consider the following formula:
\begin{lemma}
\[
\langle L^J  \mid S^K \rangle = \theta(K- J)
\]
\end{lemma}
\begin{proof}
Let's start with the Cauchy identity and project monomials onto complete:
\begin{align*}
& \sum_I \langle M^{I}(X) \mid S^K (X) \rangle S^I(Y) = \sum_J \langle L^J (X) \mid S^K(X) \rangle R^J(Y) \quad
\text{ therefore } \\
& S^K(A) = \sum_J \langle L^J (X) \mid S^K(X) \rangle R^J(A) \\
\intertext{ comparison  this with the expansion of complete in ribbon Schur:}
& S^K(A) = \sum_{ J \preceq K } R^J
\end{align*}
concludes the proof. \qedhere
\end{proof}
\begin{proof}[Proof of Proposition \ref{prop:duality-fundamental-ribbon}]
\begin{align*}
& \langle L^I \mid R^J \rangle \stackrel{\text{ by definition of ribbon Schur}}{=}
\langle L^I | \sum_{ K \preceq J } (-1)^{\ell(J) - \ell(K)} S^K
\rangle = \sum_{J \succeq K \succeq I}  (-1)^{\ell(J) - \ell(K)} = \delta_{IJ}
\end{align*}
\qedhere
\end{proof}
\subsection{Properties of the Pairing: Isometry.}
One of the pleasing properties of this scalar product it shares with the Hall scalar product in \textbf{\textsc{Sym}} is that
\begin{state}
\label{prop-isometry}
The involution $ \omega $ is an isometry, i.e.
\begin{equation}
\label{eq:forgotten-elementary-pairing}
\langle \omega( M^I ) | \omega \left( S^J \right) \rangle = \delta_{IJ}
\end{equation}
\end{state}
\begin{proof}[Proof of Proposition \ref{prop-isometry}]
\begin{align*}
& \langle \omega( M^I ) \mid \omega \left( S^J \right) \rangle =
\left \langle (-1)^{\ell(I) - |I|} F^{\overline{I}} \mid \Lambda^{\overline{J}} \right \rangle  \Rightarrow \\
&\intertext{ by expansion of forgotten in monomial and elementary in complete}
& \Rightarrow (-1)^{\ell(I) - |I|} \left \langle \sum_{K \preceq \overline{I} } M^K \mid
 \sum_{ T \succeq \overline{J}} (-1)^{\ell(T) - |T |} S^T \right \rangle =
  (-1)^{\ell(I) - |I|}  \sum_{K \preceq \overline{I} }  \sum_{ T \succeq \overline{J}} (-1)^{\ell(T) - |T|}
  \delta_{K T} =\\
& =  (-1)^{\ell(I) } \theta( \overline{I} - \overline{J}) \sum_{\overline{I} \succeq K \succeq \overline{J}}
 (-1)^{\ell(K)} = (-1)^{\ell(I) } \theta( \overline{I} - \overline{J})
 \sum_{\overline{I} \succeq K \succeq \overline{J}} (-1)^{\ell(K)} =
 (-1)^{\ell(I) } \theta( \overline{I} - \overline{J}) (-1)^{\ell(\overline{J})} \delta_{\overline{I}
 \overline{J}} = \delta_{IJ}
\end{align*}
\end{proof}
Since the scalar product is an isometry, one gets an easy proof of Proposition \ref{prop-duality-fundamental}:
\begin{proof}
Expand $ \omega(L^I) $ in the fundamental basis and project onto ribbon Schur basis:
\[
 \omega(L^I) = \sum_{K} c_{IK} L^K
 \]
 We have to see that in this expansion the only nonzero coefficient is $ c_{II^{\widetilde{}}} = 1 $. Indeed,
\begin{align*}
& 1 = \left \langle L^I \mid R^I \right \rangle = \left \langle \omega(L^I) \mid R^{I^{\widetilde{}}} \right \rangle =
\left \langle \sum_{K} c_{IK} L^K \mid R^{I^{\widetilde{}}} \right \rangle = c_{II^{\widetilde{}}} \\
& 0 = \left \langle L^I \mid R^J \right \rangle = \left \langle \omega(L^I) \mid R^{J^{\widetilde{}}} \right \rangle =
\left \langle \sum_{K} c_{IK} L^K \mid R^{J^{\widetilde{}}} \right \rangle = c_{IJ^{\widetilde{}}} \quad \forall J \neq I
\end{align*}
\qedhere
\end{proof}
\subsection{Scalar product between various bases.}
One can deduce values of the scalar between different bases making use of transition formulas between them
and applying involution $ \omega $. \\
For instance,
\[
\langle M^I | R^K \rangle = \langle M^I | \sum_{ K \succeq J} (-1)^{\ell(K) - \ell(J)} S^J \rangle =
\sum_{ K \succeq J} (-1)^{\ell(K) - \ell(J)} \delta_{IJ} = (-1)^{\ell(K) - \ell(I)}\theta(K - I)
\]
To write down the formula for pairing between power sums, which is particularly interesting as they are not orthogonal in contrast to the classical theory,
I need to recall some more definitions from \cite{GL}.
Denote the last part of a composition $ I = (i_1, \ldots, i_k)$ by
\[
lp(I) = i_k
\]
and let $J $ be a composition such that $ J \succeq I $. Let then
$J = (J_1, . . . , J_m) $ be the unique decomposition of $J$ into compositions $(J_i)_{i=1,m} $ such that
$|J_p| = i_p, p = 1, . . . ,m$. Define
\[
lp(J, I) = \prod_{i=1}^m lp(J_i)
\]
Then
\begin{state}
\label{eq:powersum-pairing}
\[
\langle \Psi^I | \Psi^J \rangle =  \sum_{J \preceq M \preceq I} (-1)^{\ell(M) - \ell(J)} lp(M,J)
\prod_{k=1}^{\ell(M)} ( \ell(M) - k + 1)^{p_k - p_{k -1}},
\]
where $ p_k $ are such that for each $ M $
\[
M = (i_1 + \ldots + i_{p_1}, i_{p_1 + 1} + \ldots + i_{p_2}, \ldots, i_{p_{k-1}+1} + \ldots + i_{p_k},
\ldots, i_{p_s } + \ldots + i_n )
\]
In particular
\begin{equation}
\label{eq:psiI-psiI-pairing}
\langle \Psi^I | \Psi^I \rangle = \left( \prod_{k=1}^{\ell(I)} i_k \right) \ell(I)!
\end{equation}
\end{state}
Before proving  Proposition \ref{eq:powersum-pairing} however, one can get the following using (\ref{eq:Psi-through-S})
\begin{state}
\begin{equation}
\label{eq:monomial-powersum-pairing}
\langle M^I | \Psi^J \rangle = (-1)^{\ell(I) - \ell(J)} lp(I,J) \theta( I - J)
\end{equation}
\end{state}
\begin{proof}
\begin{align*}
& \langle M^I | \Psi^J \rangle =
 \langle M^I | \sum_{K \succeq J} (-1)^{\ell(K) - \ell(J)}
lp(K,J) S^K \rangle =  \sum_{K \succeq J} (-1)^{\ell(K) - \ell(J)}
lp(K,J) \langle M^I | S^K \rangle = \\
& = \sum_{K \succeq J} (-1)^{\ell(K) - \ell(J)} lp(K,J) \delta_{IK} =(-1)^{\ell(I) - \ell(J)} lp(I,J)
\theta( I - J)
\qedhere
\end{align*}
\end{proof}
\begin{proof}[Proof of Proposition \ref{eq:powersum-pairing}.]
Using the expansion of power sums in the monomial basis (\ref{eq:Psi-through-primitive}) one has
\begin{equation}
\label{eq:Psi-through-primitive}
\Psi^I = \sum_{J \preceq I} \prod_{k=1}^{\ell(J)} ( \ell(J) - k + 1)^{p_k - p_{k -1}} M^J
\end{equation}
and the pairing between monomials and power sums above (\ref{eq:monomial-powersum-pairing}) I get:
\begin{align*}
& \langle \Psi^I | \Psi^J \rangle = \langle
\sum_{M \preceq I} \prod_{k=1}^{\ell(M)} ( \ell(M) - k + 1)^{p_k - p_{k -1}} M^M | \Psi^J \rangle = \\
& = \sum_{M \preceq I} \prod_{k=1}^{\ell(M)} ( \ell(M) - k + 1)^{p_k - p_{k -1}}
(-1)^{\ell(M) - \ell(J)} lp(M,J)
\theta( M - J) = \\
& = \sum_{J \preceq M \preceq I} (-1)^{\ell(M) - \ell(J)} lp(M,J)
\prod_{k=1}^{\ell(M)} ( \ell(M) - k + 1)^{p_k - p_{k -1}}
\end{align*}
In particular, if $ I = J $, the sum reduces to one term $ M = I = J $, $ p_k = k $, and $ lp(I,I) =
\prod_{k=1}^{\ell(I)} i_k $
\[
 \langle \Psi^I | \Psi^I \rangle = lp(I,I)
\prod_{k=1}^{\ell(I)} ( \ell(I) - k + 1)^{p_k - p_{k -1}} = \left( \prod_{k=1}^{\ell(I)} i_k \right) \ell(I)!
\]
\qedhere
\end{proof}
\begin{corollary}
Since the involution is an isometry of the scalar product
\begin{equation}
\langle \Psi^{\overline{I}} | \Psi^{\overline{J}} \rangle= (-1)^{\ell(I) - \ell(J)} \langle \Psi^I | \Psi^J \rangle
\end{equation}
\end{corollary}
\begin{proof}
Apply involution $ \omega $ to
\[
\langle \Psi^{\overline{I}} | \Psi^{\overline{J}} \rangle = \langle \omega (\Psi^{\overline{I}}) | \omega( \Psi^{\overline{J}}) \rangle =
\langle (-1)^{|I| - \ell(I)} \Psi^I | (-1)^{|J| - \ell(J)} \Psi^J \rangle = (-1)^{\ell(I) - \ell(J)} \langle \Psi^I | \Psi^J \rangle
\]
\qedhere
\end{proof}
\section{A Noncommutative Identity.}
To further emphasize the extent to which the commutative theory can be extended to the noncommutative setting, I will present a
noncommutative analog of the following identity.
In  the Exercise 10, Ch. I, \S 5 of \cite{M}, it is shown that
\begin{equation}
\label{eq:identity-commutative}
\sum_{|\la|=n} X^{\ell(\la) - 1} m_{\la} = \sum_{k=0}^{n -1} s_{n - k, 1^k} \left( X - 1 \right)^k
\end{equation}

The identity (\ref{eq:identity-commutative}) has the following noncommutative analog:
\begin{state}
\begin{equation}
\label{eq:identity-noncommutative}
\sum_{|I|=n} X^{\ell(I)-1} M^I = \sum_{k=0}^{n - 1} R^{1^k n - k} (X - 1)^k
\end{equation}
\end{state}
This is a generalization of two known identities: Corollary 3.14, p. 21 of \cite{GL}
\[
\Psi_n = \sum_{k=0}^{n-1} (-1)^k R^{1^k, n-k}
\]
(at $ X =0 $) and
\[
\sum_{|I|=n} L^I = \sum_{k=0}^{n-1} R^{1^k n -k}
\]
(at $ X = 2 $) stated in \cite{U} in the context of quasi-symmetric functions.
\section{Monomial Symmetric Functions as a Sum of Commutative Images of Noncommutative Monomials.}
\label{sec-monomial}
I now turn to implications of identification of primitive monomials for commutative symmetric functions.

When entries in a matrix commute with each other, quasideterminants become (up to a sign) ratios of the determinant of
the matrix to its principal minor (see \cite{GL}). In particular, denoting the commutative version of the
primitive monomial function $ M^I $ by $ m_I $
\begin{align*}
& m_I =  \frac{(-1)^{(n - 1) + (n - 1)}}{n}\frac{
\begin{vmatrix}
p_{i_n} & 1  & \dots & 0& 0\\
p_{ i_{n -1} + i_n} & p_{i_{n - 1}}  & \ldots & 0 & 0 \\
\vdots &  \vdots & \vdots  &  \vdots & \vdots \\
p_{i_2 + \ldots + i_n} & \ldots  & \ldots&  p_{i_2}& n - 1 \\
p_{i_1 + \ldots + i_n} & \ldots  & \ldots &p_{i_1 + i_2}& p_{i_1}
\end{vmatrix}}
{ \begin{vmatrix}
1 & 0 & \dots & 0& 0\\
 p_{i_{n - 1}} & 2 & \ldots & 0 & 0 \\
  \vdots & \vdots & \vdots &  \vdots & \vdots \\
p_{i_2 + \ldots + i_n} & \ldots & \ldots&  p_{i_2}& n - 1
\end{vmatrix}} = \frac1{ n!}
\begin{vmatrix}
p_{i_n} & 1  & \dots & 0& 0\\
p_{ i_{n -1} + i_n} & p_{i_{n - 1}}  & \ldots & 0 & 0 \\
\vdots &  \vdots & \vdots  &  \vdots & \vdots \\
p_{i_2 + \ldots + i_n} & \ldots  & \ldots&  p_{i_2}& n - 1 \\
p_{i_1 + \ldots + i_n} & \ldots  & \ldots &p_{i_1 + i_2}& p_{i_1}
\end{vmatrix}
\end{align*}
The main purpose of this section is to present evidence indicating
 that $ M^I $ are indeed proper noncommutative analogs of the monomial symmetric function. Denote the commutative image of the noncommutative
 monomial symmetric function $ M^I $ by $ m_I $, then
the following is true:
\begin{state}
\label{mon-det}
\begin{equation}
\label{eq-mon-det}
\widetilde{m}_{\boldsymbol{\mu}} = \sum_{\mathfrak{S_n}} m_{I},
\end{equation}
where the sum is over all permutations of the composition $I$, $ \boldsymbol{\mu} $ is the partition obtained
by ordering parts of $ I$, and $ n = \ell(I)$.  \\
Or, equivalently,
\[
 m_{\boldsymbol{\mu}} = \sum_{I \sim \boldsymbol{\mu} } m_{I},
\]
\end{state}
where the sum is over all distinct composition obtained by permutating parts of $ \boldsymbol{\mu} $.\\
In fact I will prove a slightly stronger (and slightly less symmetric) statement (\ref{eq-mondet-1}). Equality
(\ref{eq-mon-det}) will be an immediate consequence of (\ref{eq-mondet-1}).

The argument requires the following well-known result (see, for instance, \cite{Na}):
\begin{lemma}{Pieri formula for monomial symmetric functions.}
\begin{equation}
\label{lemma}
p_r \cdot \widetilde{m}_{\ka} = \sum_{\la } \widetilde{m}_{\la},
\end{equation}
where the sum is over all $\la $ such that $\la / \ka  = (r) $, i.e a skew diagram $ \la / \ka$ is a row with the
length $r$.
\end{lemma}
\begin{state}
Consider a partition $ \boldsymbol{\mu}=( \mu_1, \ldots, \mu_n) $ and the augmented monomial symmetric
function corresponding to this partition $ \widetilde{m}_{\boldsymbol{\mu}} $. Pick an arbitrary part $ \mu_j $ of the
partition $  \boldsymbol{\mu} $, then
\begin{equation}
\label{eq-mondet-1}
\widetilde{m}_{\boldsymbol{\mu}} = n \sum_{\mathfrak{S_{n-1}}}  m_{I},
\end{equation}
where the sum is over all possible compositions that can be
obtained from $  \boldsymbol{\mu} $ by permuting  all  parts while keeping $ \mu_j $ fixed.
\end{state}
Observe
that Proposition \ref{mon-det} immediately follows from (\ref{eq-mondet-1}) using the latter $ n $ times
for each part of $ \boldsymbol{\mu} $ in turn and averaging the results.
\begin{proof}
The proof will proceed by induction in number of parts of $  \boldsymbol{\mu} $.
When $ \boldsymbol{\mu} $ has one part, $ m_{(r)}= p_r = m_r $ and there is nothing to prove.

Suppose that the statement of the equation (\ref{eq-mondet-1}) is true for $ \boldsymbol{\mu} $, such that
$ \ell(\boldsymbol{\mu}) = n $.

I have to prove that (\ref{eq-mondet-1}) is true for a partition with $ n + 1 $ parts.
Denote the partition resulting in adding of part $ (r) $ to $ \boldsymbol{\mu}  $ by   $ \boldsymbol{\mu} \oplus (r)$.
Then the objective can be stated as follows:
\[
\widetilde{m}_{\boldsymbol{\mu} \oplus (r)} = (n +1) \sum_{\mathfrak{S_{n}}}  m_{I}
 \]

Consider the product of the power sum $ p_r $ and $ \widetilde{m}_{\boldsymbol{\mu}} $ with
$ \ell(\boldsymbol{\mu}) = n $ so that (\ref{eq-mondet-1}) and therefore (\ref{eq-mon-det}) are valid:
\begin{align*}
& p_r \cdot \widetilde{m}_{\boldsymbol{\mu}} \stackrel{ \text{ by } (\ref{eq-mon-det})}{=}
\sum_{\mathfrak{S_n}} p_r \cdot m_{I} \stackrel{ \text{ by } (\ref{eq-pieri})}{=}
\sum_{\mathfrak{S_n}} \left( ( n + 1) m_{r, \sigma(\mu_1), \ldots, \sigma(\mu_n)} +
 n m_{r + \sigma(\mu_1), \ldots, \sigma(\mu_n)} \right) = \\
 & \stackrel{ \text{by induction hypothesis} }{=}
 \sum_{\mathfrak{S_n}} ( n + 1) m_{r, \sigma(\mu_1), \ldots, \sigma(\mu_n)} +
 \sum_j \widetilde{m}_{\mu_1, \ldots, r + \mu_j, \ldots, \mu_n},
\end{align*}
i.e.
\[
 \sum_{\mathfrak{S_n}} ( n + 1) m_{r, \sigma(\mu_1), \ldots, \sigma(\mu_n)} =
 p_r \cdot \widetilde{m}_{\boldsymbol{\mu}} -
 \sum_j \widetilde{m}_{\mu_1, \ldots, r + \mu_j, \ldots, \mu_n}
 \stackrel{ \text{ by } (\ref{lemma})}{=} \widetilde{m}_{\boldsymbol{\mu} \oplus (r)}
\]
\qedhere
\end{proof}
\section{Summary and Open Questions.}
Once one accepts that there is a meaningful analog of noncommutative monomial $ M^I $ and fundamental symmetric functions $ L^I $,
 it is natural to inquire, first of all, whether there exist $ q - $ and $ q, t - $ deformations that interpolate between
$ M^I $ and Schur-like bases: fundamental and/or ribbon Schur functions just like Hall-Littlewood and Macdonald polynomials do in the
commutative case.
 These functions would necessarily differ from existing $ q - $ and $ q, t - $ of \cite{H} and \cite{Z} as these polynomials do not
 interpolate. \\
 Secondly, the pairing between fundamental quasi-symmetric  functions and ribbon Schur functions has a nice representation-theoretic
 explanation \cite{Th}. It would be interesting to see if there is a representation-theoretic interpretation of noncommutative Cauchy
 identity and the meaning of Kostka-Gessel numbers in that context.

\end{document}